\documentclass[11pt]{amsart}
\usepackage[utf8]{inputenc}
\usepackage{amsmath,amssymb,amsthm}
\usepackage{geometry}
\newtheorem{theorem}{THEOREM}
\newtheorem{lemma}{LEMMA}

\title{A Criterion for classifying Elliptic curves that are unsolvable in the set integers $\mathbb{Z}^+$}
\author{Shazali Abdalla Fadul}
\date{}

\begin{document}
\maketitle

\begin{abstract}
In this work, we establish a criterion for classifying elliptic curves that do not have solutions in the set of positive integers. In particular, we prove that: Let
\[
E: y^2 = Ax^3 + Bx^2 + Cx + D
\]
be an elliptic curve where $A,B,C,D,k \in \mathbb{Z}$ are the coefficients of the curve $E$, and $p$ is a prime number where $p \equiv 1 \pmod 4$ and $E(\mathbb{Z}^+)$ has positive integer points, then if
\[
-3A + C \equiv -1 \pmod 4
\]
\[
A - B + C \equiv -1 \pmod 4
\]
\[
p^2 = A - B + C - D
\]
then there are no positive solutions for the curve $E$, where $E(\mathbb{Z}^+) = \emptyset$ if $x>1, y \ge 1$.

Furthermore, we generalize this criterion by incorporating techniques from linear algebra. Specifically, we introduce a matrix construction to establish the nonexistence of positive integer solutions for a family of elliptic curves.

\end{abstract}

\section{Introduction}

Elliptic curves described by the equation $E: y^2 = x^3 + Ax + B$ are one of the most important topics of study in algebraic and arithmetic geometry. Mathematicians study the solutions of the elliptic curve in the set of integers and rational numbers $\mathbb{Z}, \mathbb{Q}$. They play a crucial role in proving new mathematical results and theorems. For example, mathematician Andrew Wiles \cite{ref14} used elliptic curves to prove Fermat's Last Theorem. Many mathematicians have studied elliptic curves, including Euler, who used algebraic methods to prove the non-existence of solutions to some elliptic curves, such as proving the non-existence of integer points on the curve $y^2 = x^3+1$, other than $y=3$ and $x=2$. There are many examples of other mathematicians. Special cases of solutions to the elliptic curve $E: y^2 = (x+a)(x^3-ax+b)$ in the set of integers have been studied by Guo Jiang, Zhao Jian, see \cite{ref25,ref26,ref18} and others.

In 2010, the elliptic curve was introduced.
\[
y^2 = x^3 - 27x - 62
\]
was studied by Wu Huaming, see \cite{ref25}, who used a method based on elementary number theory. All trend points have been reached.

Also, using the same ideas, the generalized case at $b=43$ and $a=2$ by Du Xian-cun, Zhao Jian-hong \cite{ref26} was studied. Also in 2017 the points of the elliptic curve $y^2=(x+6)(x^2-6x+19)$ were studied, see \cite{ref18}, studied by Guo Jing, arriving at the final conclusion that there are no solutions in the set of integers for that elliptic curve. In addition, there are other scholars who studied the properties and characteristics of solutions of elliptic curves, including Louis Mordell \cite{ref15}, who proved that the set of points on the curve $E$ forms a finitely generated Abelian group, see \cite{ref16} for further discussion. Mazur also proved (Mazur's theorem, see \cite{ref6}), in 1977, that the torsion subgroup of the group of rational points $E(\mathbb{Q})$ on an elliptic curve must be one of 15 groups. In particular $E(\mathbb{Q})_{\mathrm{tors}}$ has order at most 16.

There are also others, namely Siegel \cite{ref11}, Nigel-Lutz \cite{ref9}, and Henry Bencarrie \cite{ref10}, Faltings \cite{ref16}, all of whom studied and developed important results regarding the properties and characteristics of elliptic curve solutions. The results of Siegel and others have provided us with many algebraic properties of elliptic curve solutions. They have answered many questions, but some questions remain unanswered: how can we prove the existence or non-existence of solutions for an elliptic curve $E$? This question has no general answer. Therefore, Du and Wu Huaming studied special cases of curves to prove the existence or non-existence of solutions. In general, all the results mentioned do not provide us with a method to prove the existence or non-existence of solutions for elliptic curves in general.

The main objective of this work is to classify elliptic curves that admit no solutions in the set of positive integers. To this end, we establish a criterion for the classification of such elliptic curves. For example, we prove that: Let
\[
E: y^2 = Ax^3 + Bx^2 + Cx + D
\]
be an algebraic curve where the coefficients of the curve are $A,B,C,D,k \in \mathbb{Z}$, $k$ is an odd number, all prime divisors of $k$ are of the form $p = 4n+1$, and $E(\mathbb{Z}^+)$ has positive integer points, then if
\[
-3A+C \equiv -1 \pmod 4, \qquad A-B+C \equiv -1 \pmod 4, \qquad k^2 = A-B+C-D
\]
holds, then there are no positive solutions for the curve $E$, i.e. $E(\mathbb{Z}^+) = \emptyset$.

We further generalize this criterion by incorporating tools from linear algebra into the proposed framework. In particular, we prove that: Let $S: Ax=b$ be a system of linear equations, where $A \in M_{4\times 6}(\mathbb{Z})$. Suppose that $S=\{s_1,s_2,\dots,s_N\}\subset \mathbb{Z}^4$ is the set of solutions of the linear system, and let $k,a\ge 1$ be positive odd integers such that every prime divisor $p$ of $k,a$ satisfies $p\equiv 1 \pmod 4$, with $u,t\in\mathbb{Z}$. If the system $S$ is given by
\[
\begin{pmatrix}
a_{11} & -a_{12} & a_{13} & -a_{14} \\
a_{21} & -a_{22} & a_{23} & 0 \\
-3a_{31} & 0 & a_{33} & 0 \\
a_{41} & -a_{42} & a_{43} & -a_{44} \\
a_{51} & -a_{52} & a_{53} & 0 \\
-3a_{61} & 0 & a_{63} & 0
\end{pmatrix}
\begin{pmatrix} c_1 \\ c_2 \\ c_3 \\ c_4 \end{pmatrix}
=
\begin{pmatrix} k^2 \\ 4u-1 \\ 4t-1 \\ a^2 \\ 4s-1 \\ 4r-1 \end{pmatrix}
\]
and the entries of the matrix satisfy $a_{11}=a_{21}=a_{31}$, $a_{12}=a_{22}$, $a_{13}=a_{23}=a_{33}$, $a_{41}=a_{51}=a_{61}$, $a_{42}=a_{52}$, and $a_{43}=a_{53}=a_{63}$, then there exist elliptic curves
\[
E_{s_i}: Y^2 = a_{11}c_1X^3+a_{12}c_2X^2+a_{13}c_3X+a_{14}c_4,
\]
\[
C_{s_i}: Y^2 = a_{41}c_1X^3+a_{42}c_2X^2+a_{43}c_3X+a_{44}c_4.
\]
If $x>1,y\ge1$ where $\forall(x,y)\in E_{s_i}(\mathbb{Z}^+), C_{s_i}(\mathbb{Z}^+)$ holds, then
\[
\{E_{s_i}(\mathbb{Z}^+)\}_{i=1}^N = \emptyset, \qquad \{C_{s_i}(\mathbb{Z}^+)\}_{i=1}^N = \emptyset.
\]

The paper is divided as follows: in the second section, some fundamental theory regarding prime numbers and their properties is established, as well as THEOREM 3. In the third section, we prove THEOREM 4, an important result concerning elliptic curves. The theorem describes the relationship of elliptic curves over the set of integers with a given algebraic equation. We also prove THEOREM 5, a strong and fundamental theorem in this article, giving a criterion for classifying elliptic curves that are not solvable in the set of positive integers. In the fourth section, we study special cases of elliptic curves using THEOREM 5. In Section 5, we generalize the criterion using linear-algebraic techniques through a matrix construction. As a consequence, we establish THEOREMS 9, 10, 11, and 12.

\section{The Method of Proof}

In this section, we establish the fundamental results on which the proof of the main results in this study relies. First, we prove THEOREM 1, a result concerning the properties of prime numbers of the form $4n-1$. We also mention Fermat's theorem on the sum of two squares. We use these results in proving THEOREM 3.

\begin{theorem}\label{thm1}
Let $n=4m-1$ and $m>0\in\mathbb{N}$. Then we find $p$ prime where $p\mid n$ and $p\equiv -1\pmod 4$, $p\le n$.
\end{theorem}

\begin{proof}
Let $n=4m-1\in\mathbb{N}$, so according to the fundamental theorem of arithmetic (see \cite[Chp.\,1, p.\,83]{ref24}, Theorems 3,4), we have $n = p_1^{e}\cdot p_2\cdots p_n$. We conclude from this that if $p_1^{e},p_2,\dots,p_n$ are all of the form $p_j=4n+1$, then $4m-1 = 4w+1 = p_1^e\cdot p_2\cdots p_n$, a contradiction. Therefore, there is at least one prime number of the form $p=4k-1$ where $n=p(4w+1)=4m-1$, so $p\le n$.
\end{proof}

\begin{theorem}[Fermat's theorem on the sum of two squares]\label{thm2}
Let $m$ be a natural number that is not a perfect square. Then
\[
m = a^2+b^2
\]
if and only if all prime factors of $m$ are not of the form $p=4n-1$.
\end{theorem}

\begin{proof}
(See proof in \cite[Chp.\,8, p.\,242]{ref24}.) Suppose all prime divisors of $m$ are not of the form $4n-1$. If $m=1$ then $m=1^2+0^2$, and if $m>1$ then $m=\prod_{i=1}^r p_i$. Now, if $p_i=1^2+1^2$ (for $p_i=2$) and if $p_i\equiv 1\pmod 4$ for all $i$, we get $p_i = a_i^2+b_i^2$ according to Fermat's theorem on the sum of two squares (see \cite[Thm.\,15]{ref24}). But
\[
p_1p_2 = (a_1^2+b_1^2)(a_2^2+b_2^2) = (a_1a_2+b_1b_2)^2+(a_1a_2-b_1a_2)^2.
\]
So by induction on $r$ we can prove that
\[
\prod_{i=1}^r p_i = a^2+b^2. \qedhere
\]
\end{proof}

\begin{theorem}\label{thm3}
Let $A,B,C,\alpha\in\mathbb{Z}$ be integers, with $\alpha$ odd positive, such that all prime divisors of $\alpha = p_1\cdot p_2\cdot p_3\cdots p_n$ are of the form $p_j=4n+1$. Then
\[
y^2+\alpha^2 \ne (x+1)f(x) \quad \text{for all } (x,y) > (1,0) \in (\mathbb{Z}^2)^+,
\]
where
\[
f(x) = A\sum_{j=0}^2 \binom{3}{j}(x+1)^{2-j}(-1)^j + B(x-1)+C,
\]
and
\[
-3A+C\equiv -1\pmod 4, \qquad A-B+C\equiv -1\pmod 4.
\]
\end{theorem}

\begin{proof}
Let $A,B,C,\alpha\in\mathbb{Z}$ be positive integers where $\alpha$ is an odd positive integer, all of whose prime divisors $\alpha=p_1\cdot p_2\cdot p_3\cdots p_n$ are of the form $p_j=4n+1$. Let
\[
y^2+\alpha^2 = (x+1)f(x), \qquad f(x)=A\sum_{j=0}^2\binom{3}{j}(x+1)^{2-j}(-1)^j+B(x-1)+C.
\]

\textbf{Case $x=4m-1$.} Assume $(x,y)>(1,0)\in(\mathbb{Z}^2)^+$ where $x=4m-1$, $y\in\mathbb{Z}^+$. Then
\[
y^2+\alpha^2 = (4m-1+1)f(4m-1) = (4m)f(4m-1).
\]
Notice that $(4m)f(4m-1)\equiv 0\pmod 4$, so $y^2+\alpha^2\equiv 0\pmod 4$. Since $\alpha$ is fixed odd, $\alpha=2c+1$. If $y=2b+1$ is also odd, then $(2b+1)^2+(2c+1)^2\equiv 2\pmod 4$, so $(2b+1)^2+(2c+1)^2\not\equiv 0\pmod4$; and if $y=2b$ is even, then $(2b)^2+(2c+1)^2\equiv 1\pmod 4$. We conclude
\begin{equation}
y^2+\alpha^2\not\equiv 0\pmod 4 \ \Rightarrow\ y^2+\alpha^2\ne (4m)f(4m-1)\ \text{for all } y\in\mathbb{Z}^+. \tag{2.1}
\end{equation}
So we conclude from (2.1):
\begin{equation}
y^2+\alpha^2 \ne (x+1)f(x) \quad\text{if } (x=4m-1,y)\in\mathbb{Z}^{2+}. \tag{2.2}
\end{equation}

\textbf{Case $x=2(2f+1)$.} Now assume $(x,y)>(1,0)\in\mathbb{Z}^{2+}$, $x=2(2f+1)$, $y\in\mathbb{Z}^+$. Then
\[
y^2+\alpha^2=(4f+2+1)f\big(2(2f+1)\big) = \big(4(f+1)-1\big) f\big(2(2f+1)\big).
\]
By Theorem \ref{thm1}, there exists a prime $q=4n-1$ dividing $4(f+1)-1=4C-1$. So $4C-1=qm$ with $q=4n-1$, hence
\begin{equation}
y^2+\alpha^2 = (qm)\,f\big(2(2f+1)\big). \tag{2.3}
\end{equation}
By Fermat's theorem (Theorem \ref{thm2}), all prime divisors of $y^2+\alpha^2$ are of the form $4n+1$; hence
\begin{equation}
y^2+\alpha^2 \ne (qm)f\big(2(2f+1)\big). \tag{2.4}
\end{equation}
Then
\begin{equation}
y^2+\alpha^2 \ne (x+1)f(x) \quad\text{if } (x=2(2f+1),y)\in\mathbb{Z}^{2+}. \tag{2.5}
\end{equation}

Now assume $(x,y)>(1,0)\in \mathbb{Z}^{2+}$, $x=2(2f)$, $y\in\mathbb{Z}^+$. Then
\begin{equation}
y^2+\alpha^2 = (4f+1)f\big(2(2f)\big), \tag{2.6}
\end{equation}
where
\begin{equation}
f\big(2(2f)\big) = A\sum_{j=0}^2\binom{3}{j}(4f+1)^{2-j}(-1)^j + B(4f-1)+C. \tag{2.7}
\end{equation}
Then
\[
A\sum_{j=0}^2\binom{3}{j}(4f+1)^{2-j}(-1)^j = A\big((4f+1)^2-3(4f+1)+3\big) = A\big((4f)^2+8f+1-12f\big)=4A(4f^2-f)+A.
\]
Let $n=A(4f^2-f)$. Then
\begin{equation}
A\sum_{j=0}^2\binom3j(4f+1)^{2-j}(-1)^j = 4n+A. \tag{2.8}
\end{equation}
From (2.6) and (2.7):
\[
f\big(2(2f)\big) = 4n+A+B(4f)-B+C = 4(n+Bf)+A-B+C.
\]
Let $k=n+Bf$. Then $f(2(2f)) = 4k+A-B+C$. Assume $A-B+C=4n-1$ for some $n\in\mathbb{Z}$. Then
\begin{equation}
f\big(2(2f)\big) = 4k+4n-1 = 4m-1. \tag{2.9}
\end{equation}
From (2.6) and (2.9):
\[
y^2+\alpha^2 = (4f+1)(4m-1).
\]
By Theorem \ref{thm1}, there exists a prime $q=4n-1$ dividing $4m-1$, so $4m-1=qw$, hence
\[
y^2+\alpha^2 = (4f+1)(wq).
\]
By Theorem \ref{thm2}, all prime divisors of $y^2+\alpha^2$ are of the form $4n+1$, so
\[
y^2+\alpha^2 \ne (4f+1)(wq).
\]
Then
\begin{equation}
y^2+\alpha^2 \ne (x+1)f(x) \quad\text{if } (x=2(2f),y)\in\mathbb{Z}^{2+}. \tag{2.10}
\end{equation}

\textbf{Case $x=4m+1$.} Now assume $(x,y)>(1,1)\in\mathbb{Z}^{2+}$, $x=4m+1$, $y\in\mathbb{Z}^+$. Then
\begin{equation}
y^2+\alpha^2 = (4m+2)f(4m+1), \tag{2.11}
\end{equation}
\begin{equation}
f(4m+1) = A\sum_{j=0}^2\binom3j(4m+2)^{2-j}(-1)^j + B(4m)+C. \tag{2.12}
\end{equation}
We have
\[
A\sum_{j=0}^2\binom3j(4m+2)^{2-j}(-1)^j = A(4m+2)^2-3A(4m+2)+3A = 4\big(4A(2m+1)^2-3Am\big)-3A.
\]
Let $n=4A(2m+1)^2-3Am$. Then
\begin{equation}
A\sum_{j=0}^2\binom3j(4m+2)^{2-j}(-1)^j = 4n-3A. \tag{2.13}
\end{equation}
From (2.12) and (2.13): $f(4m+1)=4(n+Bm)-3A+C$. Let $k=n+Bm$; then $f(4m+1)=4k-3A+C$. Assume $-3A+C=4m-1$ for some $m\in\mathbb{Z}$. Then
\begin{equation}
f(4m+1) = 4k-3A+C = 4k+4m-1 = 4(k+m)-1 = 4w-1. \tag{2.14}
\end{equation}
From (2.11) and (2.14): $y^2+\alpha^2=(4m+2)(4w-1)$. By Theorem \ref{thm1}, there exists a prime $q=4n-1$ dividing $4w-1$, so $4w-1=qm$, hence $y^2+\alpha^2=(4m+2)qm$. By Theorem \ref{thm2}, all prime divisors of $y^2+\alpha^2$ are of the form $4n+1$, so
\begin{equation}
y^2+\alpha^2 \ne (x+1)f(x) \quad\text{if } (x=4m+1,y)\in\mathbb{Z}^{2+}. \tag{2.15}
\end{equation}

So according to equations (2.2), (2.5), (2.10), (2.15), we have: if $(x,y)=(2m,4m,4m-1,4m+1)\in(\mathbb{Z}^2)^+$ then $y^2+\alpha^2\ne(x+1)f(x)$. So it is easy to conclude:
\[
\text{if } (x,y)>(1,0)\in(\mathbb{Z}^2)^+ \text{ then } y^2+\alpha^2 \ne (x+1)f(x),
\]
where
\[
f(x) = A\sum_{j=0}^2\binom3j(x+1)^{2-j}(-1)^j + B(x-1)+C,
\]
and
\[
-3A+C\equiv -1\pmod4, \qquad A-B+C\equiv -1\pmod4. \qedhere
\]
\end{proof}

\section{The Elliptic Curve in $\mathbb{Z}^+$}

In this section, we establish the main results of the study, namely THEOREM 5, which determines the elliptic curves and the conditions that must be satisfied by their coefficients in order to be insoluble in the set of positive integers. In the proof, we rely on THEOREM 3 and THEOREM 4.

\begin{theorem}\label{thm4}
Let $E: y^2 = Ax^3+Bx^2+Cx+D$ be an algebraic curve where $A,B,C,D\in\mathbb{Z}$ are coefficients of $E$ and $k$ is an odd number, all of whose prime divisors are of the form $p=4n+1$, and $E(\mathbb{Z}^+)$ has positive integer points. Then
\[
\text{if } y^2+k^2 \ne (x+1)f(x) \text{ then } (x,y)\notin E(\mathbb{Z}^+),
\]
where
\[
f(x) = A\sum_{j=0}^2\binom3j(x+1)^{2-j}(-1)^j+B(x-1)+C, \qquad k^2 = A-B+C-D.
\]
\end{theorem}

\begin{proof}
Suppose the equation is given and its coefficients satisfy
\begin{equation}
k^2 = A-B+C-D, \tag{3.1}
\end{equation}
\begin{equation}
y^2+k^2 \ne (x+1)\left(A\sum_{j=0}^2\binom3j(x+1)^{2-j}(-1)^j+B(x-1)+C\right). \tag{3.2}
\end{equation}
We have
\[
y^2+k^2 \ne \sum_{j=0}^2\binom3j(x+1)^{3-j}(-1)^j + B(x^2-1)+C(x+1).
\]
Then
\[
y^2+k^2 \ne A\big((x+1-1)^3-(-1)^3\big)+B(x^2-1)+C(x+1).
\]
Hence
\begin{equation}
y^2 \ne Ax^3+Bx^2+Cx+A-B+C-k^2. \tag{3.3}
\end{equation}
From equation (3.1) we have $k^2+D=A-B+C$, so from (3.1) and (3.3):
\[
y^2 \ne Ax^3+Bx^2+Cx+D. \tag{3.4}
\]
This means $(x,y)\notin E(\mathbb{Z}^+)$. It follows from (3.1), (3.2), (3.4) that
\[
\text{if } y^2+k^2\ne(x+1)f(x) \text{ then } (x,y)\notin E(\mathbb{Z}^+),
\]
where $f(x)$ and $k^2$ are as above. \qedhere
\end{proof}

\begin{theorem}\label{thm5}
Let $E: y^2=Ax^3+Bx^2+Cx+D$ be an algebraic curve where $A,B,C,D,k\in\mathbb{Z}$ are coefficients of $E$, $k$ is an odd number, all of whose prime divisors are of the form $p=4n+1$, and $E(\mathbb{Z}^+)$ has positive integer points. Then if
\[
-3A+C\equiv -1\pmod4,\qquad A-B+C\equiv-1\pmod4, \qquad k^2=A-B+C-D
\]
hold, and if $x>1,\ y\ge1$ where $\forall (x,y)\in E(\mathbb{Z}^+)$, then there are no positive solutions for the curve $E$, i.e.\ $E(\mathbb{Z}^+)=\emptyset$.
\end{theorem}

\begin{proof}
By Theorem \ref{thm4}, if
\begin{equation}
y^2+k^2\ne (x+1)f(x) \text{ then } (x,y)\notin E(\mathbb{Z}^+), \qquad k^2=A-B+C-D, \tag{3.5}
\end{equation}
and by Theorem \ref{thm3},
\begin{equation}
y^2+\alpha^2 \ne (x+1)f(x) \quad\text{for all } (x,y)>(1,0)\in(\mathbb{Z}^2)^+, \tag{3.6}
\end{equation}
provided $-3A+C\equiv-1\pmod4$ and $A-B+C\equiv-1\pmod4$. Since $E(\mathbb{Z}^+)\subset(\mathbb{Z}^2)^+$, from (3.5) and (3.6), if $x>1,y\ge1$ where $\forall(x,y)\in E(\mathbb{Z}^+)$ holds, and if
\[
-3A+C\equiv-1\pmod4,\quad A-B+C\equiv-1\pmod4,\quad k^2=A-B+C-D,
\]
then
\begin{equation}
y^2+k^2 \ne (x+1)f(x) \quad\text{for all } (x,y)\in E(\mathbb{Z}^+). \tag{3.7}
\end{equation}
From (3.5) and (3.7) we have $E(\mathbb{Z}^+)=\emptyset$. Then there are no positive solutions for the curve $E$ where $E(\mathbb{Z}^+)=\emptyset$ if $(x,y)>(1,0)$. \qedhere
\end{proof}

\section{Special Cases of Elliptic Curves}

In this section, we use Theorem \ref{thm5} to prove results about solutions of elliptic curves in the set of positive integers, examining special cases and classifying elliptic curves that are unsolvable in the set of integers.

\begin{theorem}\label{thm6}
Let $E: y^2=x^3+Bx^2+Cx+D$ be an algebraic curve where $B,C,D,k\in\mathbb{Z}$ are coefficients of $E$, $k$ odd, all prime divisors of the form $p=4n+1$, and $E(\mathbb{Z}^+)$ has positive integer points. Then if
\[
-3+C\equiv-1\pmod4,\qquad 1-B+C\equiv-1\pmod4,\qquad k^2=1-B+C-D
\]
hold, then there are no positive solutions for the curve $E$, i.e. $E(\mathbb{Z}^+)=\emptyset$ if $x>1,y\ge1$.
\end{theorem}
\begin{proof}
Let $A=1$ in Theorem \ref{thm5}.
\end{proof}

\begin{theorem}\label{thm7}
Let $E: y^2=x^3+Cx+D$ be an algebraic curve where $C,D,k\in\mathbb{Z}$ are coefficients of $E$, $k$ odd, all prime divisors of the form $p=4n+1$, and $E(\mathbb{Z}^+)$ has positive integer points. Then if
\[
-3+C\equiv-1\pmod4,\qquad 1+C\equiv-1\pmod4,\qquad k^2=1+C-D
\]
hold, then there are no positive solutions for the curve $E$, i.e. $E(\mathbb{Z}^+)=\emptyset$ if $x>1,y\ge1$.
\end{theorem}
\begin{proof}
Let $B=0$ in Theorem \ref{thm6}.
\end{proof}

\begin{lemma}\label{lem1}
Let $E: y^2=x^3+(4t+2)x+D$ be an algebraic curve where $D,k\in\mathbb{Z}$ are coefficients of $E$, $k$ odd, all prime divisors of the form $p=4n+1$, and $E(\mathbb{Z}^+)$ has positive integer points. Then if
\[
k^2 = 4t-D+3
\]
holds, then there are no positive solutions for the curve $E$, i.e. $E(\mathbb{Z}^+)=\emptyset$ if $x>1,y\ge1$.
\end{lemma}
\begin{proof}
Let $C=4t+2$ in Theorem \ref{thm7}.
\end{proof}

\begin{theorem}\label{thm8}
Let $E: y^2=Ax^3+Bx^2+Cx+D$ be an algebraic curve where $A,B,C,D,k\in\mathbb{Z}$ are coefficients of $E$, $k$ odd, all prime divisors of the form $p=4n+1$, and $E(\mathbb{Z}^+)$ has positive integer points. Then if
\[
-3A+C\equiv-1\pmod4,\qquad A-B+C\equiv-1\pmod4,\qquad 1=A-B+C-D
\]
and $x>1,y\ge1$ where $\forall(x,y)\in E(\mathbb{Z}^+)$ holds, then there are no positive solutions for the curve $E$.
\end{theorem}
\begin{proof}
Let $k=1$ in Theorem \ref{thm5}.
\end{proof}

\begin{lemma}\label{lem2}
Let $E: y^2=Ax^3+Bx^2+Cx+D$ be an algebraic curve where $A,B,C,D\in\mathbb{Z}$ are coefficients of $E$, $p$ a prime number with $p\equiv1\pmod4$, and $E(\mathbb{Z}^+)$ has positive integer points. Then if
\[
-3A+C\equiv-1\pmod4,\qquad A-B+C\equiv-1\pmod4,\qquad p^2=A-B+C-D
\]
hold, then there are no positive solutions for the curve $E$.
\end{lemma}
\begin{proof}
Let $k=p$, with $p$ prime of the form $p=4n+1$, in Theorem \ref{thm5}.
\end{proof}

\section{Matrix Construction and Solutions to the Curve}

In this section, the criterion is further generalized by means of linear-algebraic tools. We prove that whenever a given system of linear equations $S$ admits a solution, it is possible to construct algebraic curves from the solutions of $S: Ax=b$, where $A\in M_{n\times m}(\mathbb{Z})$, in such a way that the resulting elliptic curve has no solutions in the set of positive integers.

\begin{theorem}\label{thm9}
Let $S: Ax=b$ be a system of linear equations, where $A\in M_{3\times4}(\mathbb{Z})$. Suppose that $S=\{s_1,s_2,\dots,s_N\}\subset\mathbb{Z}^4$ is the set of solutions of the linear system, and let $k$ be a positive odd integer such that every prime divisor $p$ of $k$ satisfies $p\equiv1\pmod4$, with $u,t\in\mathbb{Z}$. If system $S$ is given by
\[
\begin{pmatrix}
a_{11}&-a_{12}&a_{13}&-a_{14}\\
a_{21}&-a_{22}&a_{23}&0\\
-3a_{31}&0&a_{33}&0
\end{pmatrix}
\begin{pmatrix}c_1\\c_2\\c_3\\c_4\end{pmatrix}
=\begin{pmatrix}k^2\\4u-1\\4t-1\end{pmatrix}
\]
and the entries satisfy $a_{11}=a_{21}=a_{31}$, $a_{12}=a_{22}$, $a_{13}=a_{23}=a_{33}$, and if $s_i=(c_1,c_2,c_3,c_4)\in S$, then there exists an elliptic curve
\[
E_{s_i}: Y^2 = a_{11}c_1X^3+a_{12}c_2X^2+a_{13}c_3X+a_{14}c_4.
\]
If $x>1,y\ge1$ where $\forall(x,y)\in E_{s_i}(\mathbb{Z}^+)$ holds, then
\[
\{E_{s_i}(\mathbb{Z}^+)\}_{i=1}^N = \emptyset.
\]
\end{theorem}

\begin{proof}
By Theorem \ref{thm5}: if $-3A+C\equiv-1\pmod4$, $A-B+C\equiv-1\pmod4$, $k^2=A-B+C-D$, and $x>1,y\ge1$ where $\forall(x,y)\in E(\mathbb{Z}^+)$ holds, then there are no positive solutions for $E$. So from Theorem \ref{thm5}, the coefficients satisfy
\[
-3A+C=4u-1 \text{ where } u\in\mathbb{Z}, \qquad A-B+C=4t-1 \text{ where } t\in\mathbb{Z},
\]
\[
A-B+C-D=k^2.
\]
Assume $A=a_{11}c_1$, $B=a_{12}c_2$, $C=a_{13}c_3$, $D=a_{14}c_4$, where $a_{ij},c_j\in\mathbb{Z}$. As a result, we get a system of linear equations:
\[
a_{11}c_1-a_{12}c_2+a_{13}c_3-a_{14}c_4=k^2,
\]
\[
a_{11}c_1-a_{12}c_2+a_{13}c_3=4t-1,
\]
\[
-3a_{11}c_1+a_{13}c_3=4u-1.
\]
Organized into matrix form $S: Ax=b$ where $A\in M_{3\times4}(\mathbb{Z})$:
\[
\begin{pmatrix}
a_{11}&-a_{12}&a_{13}&-a_{14}\\
a_{21}&-a_{22}&a_{23}&0\\
-3a_{31}&0&a_{33}&0
\end{pmatrix}
\begin{pmatrix}c_1\\c_2\\c_3\\c_4\end{pmatrix}
=\begin{pmatrix}k^2\\4u-1\\4t-1\end{pmatrix}.
\]
If the entries satisfy $a_{11}=a_{21}=a_{31}$, $a_{12}=a_{22}$, $a_{13}=a_{23}=a_{33}$, and there exists a solution $s_1=(c_1,c_2,c_3,c_4)$, then there exists an elliptic curve
\[
E_{s_1}: Y^2=a_{11}c_1X^3+a_{12}c_2X^2+a_{13}c_3X+a_{14}c_4.
\]
The existence of a solution means the system satisfies the conditions of Theorem \ref{thm5} on the coefficients of $E_{s_i}$; therefore $E_{s_1}(\mathbb{Z}^+)=\emptyset$. Assume there is another solution $s_2=(c_1,c_2,\dots,c_4)$, giving a curve $E_{s_2}: Y^2=a_{11}c_1X^3+a_{12}c_2X^2+a_{13}c_3X+a_{14}c_4$; then likewise $E_{s_2}(\mathbb{Z}^+)=\emptyset$. In general, if $S=\{s_1,\dots,s_N\}\subset\mathbb{Z}^4$ represents all solutions of the system, then $\{E_{s_i}(\mathbb{Z}^+)\}_{i=1}^N=\emptyset$. \qedhere
\end{proof}

\begin{theorem}\label{thm10}
Let $S: Ax=b$ be a system of linear equations, where $A\in M_{4\times3}(\mathbb{Z})$. Suppose $S=\{s_1,\dots,s_N\}\subset\mathbb{Z}^4$ is the set of solutions, and let $k$ be a positive odd integer such that every prime divisor $p$ of $k$ satisfies $p\equiv1\pmod4$, with $u,t\in\mathbb{Z}$. If system $S$ is given by
\[
\begin{pmatrix}
a_{11}&-a_{12}&a_{13}&-a_{14}&0&0&0\\
a_{21}&-a_{22}&a_{23}&0&4a_{25}&4a_{26}&4a_{37}\\
-3a_{31}&0&a_{33}&0&4a_{35}&4a_{36}&4a_{37}
\end{pmatrix}
\begin{pmatrix}c_1\\c_2\\c_3\\c_4\\c_5\\c_6\\c_7\end{pmatrix}
=\begin{pmatrix}k^2\\4u-1\\4t-1\end{pmatrix}
\]
and the entries satisfy $a_{11}=a_{21}=a_{31}$, $a_{12}=a_{22}$, $a_{13}=a_{23}=a_{33}$, and $s_i=(c_1,\dots,c_7)\in S$, then there exists an elliptic curve
\[
E_{s_i}: Y^2=a_{11}c_1X^3+a_{12}c_2X^2+a_{13}c_3X+a_{14}c_4.
\]
If $x>1,y\ge1$ where $\forall(x,y)\in E_{s_i}(\mathbb{Z}^+)$ holds, then $\{E_{s_i}(\mathbb{Z}^+)\}_{i=1}^N=\emptyset$.
\end{theorem}

\begin{proof}
According to Theorem \ref{thm9}, if
\[
a_{11}c_1-a_{12}c_2+a_{13}c_3-a_{14}c_4=k^2,
\]
\[
a_{11}c_1-a_{12}c_2+a_{13}c_3=4t-1,
\]
\[
-3a_{11}c_1+a_{13}c_3=4u-1,
\]
and $x>1,y\ge1$ where $\forall(x,y)\in E(\mathbb{Z}^+)$ holds, then there are no positive solutions for the curve $E_{s_1}$. Note $t,u\in\mathbb{Z}$; let
\[
4u-1 = 4r-4a_{25}c_5-4a_{26}c_6-4a_{17}c_7-1,
\]
\[
4t-1 = 4s-4a_{35}c_5-4a_{36}c_6-4a_{37}c_7-1.
\]
Then we have the new equations
\[
a_{11}c_1-a_{12}c_2+a_{13}c_3-a_{14}c_4=k^2,
\]
\[
a_{11}c_1-a_{12}c_2+a_{13}c_3+4a_{35}c_5+4a_{36}c_6+4a_{37}c_7=4s-1,
\]
\[
-3a_{11}c_1+a_{13}c_3+4r+4a_{25}c_5+4a_{26}c_6+4a_{17}c_7=4r-1.
\]
We conclude that this system satisfies the conditions of Theorem \ref{thm5}, so it is possible to construct an elliptic curve with $E_{s_1}(\mathbb{Z}^+)=\emptyset$. Arranging the linear equations as a matrix $S: Ax=b$ where $A\in M_{7\times3}(\mathbb{Z})$:
\[
\begin{pmatrix}
a_{11}&-a_{12}&a_{13}&-a_{14}&0&0&0\\
a_{21}&-a_{22}&a_{23}&0&4a_{25}&4a_{26}&4a_{37}\\
-3a_{31}&0&a_{33}&0&4a_{35}&4a_{36}&4a_{37}
\end{pmatrix}
\begin{pmatrix}c_1\\c_2\\c_3\\c_4\\c_5\\c_6\\c_7\end{pmatrix}
=\begin{pmatrix}k^2\\4u-1\\4t-1\end{pmatrix}.
\]
If the system has a solution $S=\{s_1,\dots,s_N\}\subset\mathbb{Z}^7$, then a curve $E_{s_i}$ can be constructed for $s_i=(c_1,\dots,c_7)$, and if $x>1,y\ge1$ where $\forall(x,y)\in E_{s_i}(\mathbb{Z}^+)$ holds, then
\[
E_{s_i}: Y^2=a_{11}c_1X^3+a_{12}c_2X^2+a_{13}c_3X+a_{14}c_4, \qquad \{E_{s_i}(\mathbb{Z}^+)\}_{i=1}^N=\emptyset. \qedhere
\]
\end{proof}

\begin{theorem}\label{thm11}
Let $S: Ax=b$ be a system of linear equations, where $A\in M_{4\times6}(\mathbb{Z})$. Suppose $S=\{s_1,\dots,s_N\}\subset\mathbb{Z}^4$ is the set of solutions, and let $k,a\ge1$ be positive odd integers such that every prime divisor $p$ of $k,a$ satisfies $p\equiv1\pmod4$, with $u,t\in\mathbb{Z}$. If system $S$ is given by
\[
\begin{pmatrix}
a_{11}&-a_{12}&a_{13}&-a_{14}\\
a_{21}&-a_{22}&a_{23}&0\\
-3a_{31}&0&a_{33}&0\\
a_{41}&-a_{42}&a_{43}&-a_{44}\\
a_{51}&-a_{52}&a_{53}&0\\
-3a_{61}&0&a_{63}&0
\end{pmatrix}
\begin{pmatrix}c_1\\c_2\\c_3\\c_4\end{pmatrix}
=\begin{pmatrix}k^2\\4u-1\\4t-1\\a^2\\4s-1\\4r-1\end{pmatrix}
\]
and the entries satisfy $a_{11}=a_{21}=a_{31}$, $a_{12}=a_{22}$, $a_{13}=a_{23}=a_{33}$, $a_{41}=a_{51}=a_{61}$, $a_{42}=a_{52}$, $a_{43}=a_{53}=a_{63}$, then there exist elliptic curves
\[
E_{s_i}: Y^2=a_{11}c_1X^3+a_{12}c_2X^2+a_{13}c_3X+a_{14}c_4,
\]
\[
C_{s_i}: Y^2=a_{41}c_1X^3+a_{42}c_2X^2+a_{43}c_3X+a_{44}c_4.
\]
If $x>1,y\ge1$ where $\forall(x,y)\in E_{s_i}(\mathbb{Z}^+),C_{s_i}(\mathbb{Z}^+)$ holds, then
\[
\{E_{s_i}(\mathbb{Z}^+)\}_{i=1}^N=\emptyset, \qquad \{C_{s_i}(\mathbb{Z}^+)\}_{i=1}^N=\emptyset.
\]
\end{theorem}

\begin{proof}
Suppose there are two elliptic curves $E_{s_i},C_{s_i}$:
\[
E_{s_i}: Y^2=a_{11}c_1X^3+a_{12}c_2X^2+a_{13}c_3X+a_{14}c_4, \qquad
C_{s_i}: Y^2=a_{41}c_1X^3+a_{42}c_2X^2+a_{43}c_3X+a_{44}c_4.
\]
By Theorem \ref{thm9}, there is a matrix structure behind these curves as shown above. If the entries satisfy $a_{11}=a_{21}=a_{31}$, $a_{12}=a_{22}$, $a_{13}=a_{23}=a_{33}$, $a_{41}=a_{51}=a_{61}$, $a_{42}=a_{52}$, $a_{43}=a_{53}=a_{63}$, and there is a solution $S=\{s_1,\dots,s_N\}\subset\mathbb{Z}^4$, then there exist $N$ elliptic curves with no positive integer solutions, by Theorem \ref{thm9}: if $x>1,y\ge1$ where $\forall(x,y)\in E_{s_i}(\mathbb{Z}^+),C_{s_i}(\mathbb{Z}^+)$ holds, then
\[
\{E_{s_i}(\mathbb{Z}^+)\}_{i=1}^N=\emptyset, \qquad \{C_{s_i}(\mathbb{Z}^+)\}_{i=1}^N=\emptyset. \qedhere
\]
\end{proof}

\begin{theorem}\label{thm12}
Let $S: Ax=b$ be a system of linear equations, where $A\in M_{4\times6}(\mathbb{Z})$. Suppose $S=\{s_1,\dots,s_N\}\subset\mathbb{Z}^4$ is the set of solutions of the linear system. If system $S$ is given by
\[
\begin{pmatrix}
a_{11}&-a_{12}&a_{13}&-a_{14}\\
a_{21}&-a_{22}&a_{23}&0\\
-3a_{31}&0&a_{33}&0\\
a_{41}&-a_{42}&a_{43}&-a_{44}\\
a_{51}&-a_{52}&a_{53}&0\\
-3a_{61}&0&a_{63}&0
\end{pmatrix}
\begin{pmatrix}c_1\\c_2\\c_3\\c_4\end{pmatrix}
=\begin{pmatrix}1\\-1\\-1\\1\\-1\\-1\end{pmatrix}
\]
and the entries satisfy $a_{11}=a_{21}=a_{31}$, $a_{12}=a_{22}$, $a_{13}=a_{23}=a_{33}$, $a_{41}=a_{51}=a_{61}$, $a_{42}=a_{52}$, $a_{43}=a_{53}=a_{63}$, and $s_i=(c_1,c_2,c_3,c_4)\in S$, then there exist elliptic curves
\[
E_{s_i}: Y^2=a_{11}c_1X^3+a_{12}c_2X^2+a_{13}c_3X+a_{14}c_4,
\]
\[
C_{s_i}: Y^2=a_{41}c_1X^3+a_{42}c_2X^2+a_{43}c_3X+a_{44}c_4.
\]
If $x>1,y\ge1$ where $\forall(x,y)\in E_{s_i}(\mathbb{Z}^+),C_{s_i}(\mathbb{Z}^+)$ holds, then
\[
\{E_{s_i}(\mathbb{Z}^+)\}_{i=1}^N=\emptyset, \qquad \{C_{s_i}(\mathbb{Z}^+)\}_{i=1}^N=\emptyset.
\]
\end{theorem}
\begin{proof}
Let $k=a=1$ and $u=t=s=r=0$ in Theorem \ref{thm11}.
\end{proof}

\vspace{1em}
\noindent Faculty of Mathematics Sciences \& Statistics\\
AL-Neelain University, Khartoum, Sudan\\
e-mail: shazlyabdullah3@gmail.com

\end{document}